\documentclass{emsprocart}

\newtheorem{theorem}{Theorem}[section]

\newtheorem{lemma}[theorem]{Lemma}

\theoremstyle{definition}

\newtheorem{remark}[theorem]{Remark}
 
\contact[rbenguri@fis.puc.cl]{Rafael Benguria, Instituto de F\'isica, Pontificia Universidad Cat\'olica de Chile, Avda. Vicu\~na Mackenna 4860, Santiago, Chile}

\contact[benguria@math.wisc.edu]{Soledad Benguria, Mathematics Department, University of Wisconsin - Madison, 480 Lincoln Dr, Madison, WI, USA}

\title[Non--existence of radial solutions of the Brezis-Nirenberg problem in $\mathbb{H}^n$]{An improved bound for the non-existence of radial solutions of the Brezis-Nirenberg problem in $\mathbb{H}^n$}

\author[R.~D.~Benguria and S.~Benguria]{Rafael~D.~Benguria and Soledad~Benguria \thanks{The work of R.B.  has been supported by Fondecyt (Chile) Projects \# 112--0836 and \#114-1155 and by the N\'ucleo Milenio en  ``F\'\i  sica Matem\'atica'', RC--12--0002 (ICM, Chile).
S.B. would like to thank the E.~Schr\"odinger Institute in Vienna for their hospitality while part of this work was being done.}}

\begin{document}

\bigskip

\centerline{\it To Pavel Exner on the occasion of his 70th birthday}

\bigskip

\begin{abstract}
Using a Rellich--Pohozaev argument and Hardy's inequality, we derive an improved bound on  the nonlinear eigenvalue for the non existence of radial solutions  of a Brezis--Nirenberg problem, with Dirichlet boundary conditions,
on a geodesic ball of $\mathbb{H}^n$, for $2<n<4$.
\end{abstract}



\maketitle

\section{Introduction}

For a long time virial theorems have played a key role in the localization of linear and nonlinear eigenvalues. In the spectral theory of Schr\"odinger Operators, 
the virial theorem has been widely used to prove  the absence of positive eigenvalues for various multiparticle quentum systems (see, e.g., \cite{We67,Si74,Ba75}). 
In 1983,  Br\'ezis and Nirenberg  \cite{BrNi83} considered the existence and nonexistence of solutions  of the nonlinear equation 
$$
-\Delta u = \lambda u + |u|^{p-1} u, 
$$
defined on a bounded, smooth domain of $\mathbb{R}^n$, $n >2$, with Dirichlet boundary conditions, where $p=(n+2)/(n-2)$ is the critical Sobolev exponent. 
In particular, they used a virial theorem, namely the Pohozaev identity \cite{Po65}, to prove the nonexistence of regular  solutions when the domain is star--shaped, 
for any $\lambda \le 0$, in any $n >2$. After the classical paper \cite{BrNi83} of Br\'ezis and Nirenberg, many people have considered extensions of this problem in different settings. In 
particular, the Br\'ezis--Nirenberg (BN) problem has been studied on bounded, smooth, domains of the hyperbolic space $\mathbb{H}^n$ (see, e.g., \cite{St02,BaKa12,GaSa14,Be16}), 
where one replaces the Laplacian by the Laplace--Beltrami operator in $\mathbb{H}^n$. Stapelkamp \cite{St02} proved the analog of the above mentioned nonexistence result of 
Br\'ezis--Nirenberg in $\mathbb{H}^n$. Namely she proved that there are no regular solutions of the BN problem for bounded, smooth,  star--shaped domains in $\mathbb{H}^n$ ($n>2$), 
if $\lambda \le n(n-2)/4$. The purpose of this manuscript is to give an improved bound on $\lambda$ for the nonexistence of radial (not necessarily positive) radial, regular solutions 
of the BN problem on geodesic balls of $\mathbb{H}^n$ for $2<n<4$ (see Theorem 2.1 below). Notice that for the case of radial solutions of the BN problem on a geodesic ball one can consider noninteger values of $n$, which can be considered just as a parameter.  

\bigskip

Consider the Brezis--Nirenberg problem 
\begin{equation}\label{eq:BrNih}
-\Delta_{\mathbb{H}^n}u = \lambda u + |u|^{p-1} u,
\end{equation}
on $\Omega \subset \mathbb{H}^n,$ where $\Omega$ is smooth and bounded, with Dirichlet boundary conditions, i.e., $u=0$ in $\partial\Omega$.
After expressing the Laplace Beltrami operator $\Delta_{\mathbb{H}^n}$ in terms of the conformal Laplacian, Stapelkamp \cite{St02} 
proved  that (\ref{eq:BrNih}) does not admit any regular solution for star-shaped domains $\Omega$ provided 
\begin{equation}\label{eq:boundlambda}
\lambda \le \frac{n(n-2)}{4}.
\end{equation}
Here, we consider the BN problem ({\ref{eq:BrNih}) for radial solutions on geodesic balls of $\mathbb{H}^n$. 
We can prove a different bound, namely the problem for radial solutions on a geodesic ball $\Omega^*$ does not admit a solution if 
\begin{equation}
\lambda \le  \frac{n^2(n-1)}{4(n+2)}
\label{eq:ourbound}
\end{equation}
for $n>2$. Our bound is better than ({\ref{eq:boundlambda}) in the radial case, if $2<n<4.$ Both bounds coincide when $n=4$. 
In the rest of this manuscript we give the proof of (\ref{eq:ourbound}).

\section{Nonexistence of solutions of the BN problem on geodesic balls in $\mathbb{H}^n$, for $2<n<4$.}

In the sequel we consider (not necessarily positive) radial solutions of the BN problem (\ref{eq:BrNih}) on geodesic balls of $\mathbb{H}^n$. In {\it radial coordinates}, 
(\ref{eq:BrNih}) can be written as
\begin{equation}\label{eq:BrNigen}
-u''(x) - (n-1)\coth(x)u'(x) = \lambda u(x) + |u|^{p-1} u(x),
\end{equation}
with $u'(0) = u(R) = 0$, where $R$ is the {\it radius} of the geodesic ball.  Here, as before, $p=(n+2)/(n-2)$. Notice that (\ref{eq:BrNigen}) makes sense also if $n$ is not an integer. 
For that reason henceforth  we consider $n \in \mathbb{R}$, with $2<n<4$. Our main result is the following
\begin{theorem}
The Boundary Value problem (\ref{eq:BrNih}), with $u'(0)=u(R)=0$,   has no regular solutions if 
$$
\lambda \le \frac{n^2(n-1)}{4(n+2)}, 
$$
for $2<n<4$.
\end{theorem}
\begin{remark}
Notice that our bound $n^2(n-1)/(4(n+2))$ is strictly bigger than $n(n-2)/4$ for $n <4$. Notice, on the other hand, that Stapelkamp's bound holds for all regular solutions, 
while our improved bound only holds for regular radial solutions. We do not know whether our bound is optimal, i.e., we do not know if there are solutions for 
$\lambda > n^2(n-1)/(4(n+2)$. In view of \cite{Be16}, there can be no positive solutions if $\lambda < \mu(n)$ (see \cite{Be16} for the definition of $\mu(n)$). It is 
important to notice that for $n=4$, we have that, 
$$
\frac{n(n-2)}{4} = \frac{n^2(n-1)}{4(n+2)} = \mu (4), 
$$
so at least our result is optimal as $n \to 4$. 
\end{remark}
\begin{proof}
We use a Rellich--Pohozaev argument \cite{Re40,Po65}. Multiplying equation (\ref{eq:BrNigen}) by $u(x)\sinh^{n-1}(x)$ and integrating, we obtain
\begin{equation}
\begin{split} & -\int_0^R u''(x)(u(x)\sinh^{n-1}(x))\, dx - (n-1)\int_0^R u(x)u'(x)\cosh(x)\sinh^{n-2}(x)\, dx\\ &= 
\lambda\int_0^R u^2 \sinh^{n-1}(x)\, dx + \int_0^R {|u(x)|}^{p+1}\sinh^{n-1}(x)\, dx.\\
\end{split}
\end{equation}
Integrating the first term by parts, we can write this equation as 
\begin{equation}\label{eq:Po1}
\int_0^R u'^2 \sinh^{n-1}(x)\, dx = \lambda \int_0^R u^2 \sinh^{n-1}(x)\, dx +\int_0^R {|u|}^{p+1}\sinh^{n-1}(x)\, dx.
\end{equation}

Now let $G(x) = \displaystyle\int_0^x \sinh^{n-1}(s)\, ds. $ Multiplying equation (\ref{eq:BrNigen}) by $u' G$ and integrating, we obtain 

\begin{equation*}\label{eq:21}\begin{split}
&-\int_0^R \left(\frac{u'^2}{2} \right)'G \, dx - (n-1)\int_0^R \coth(x)u'^2G\, dx\\&  = \lambda \int_0^R\left(\frac{u^2}{2}\right)'G\, dx + \int_0^R \left(\frac{{|u|}^{p+1}}{p+1}\right)'G\, dx. \\
\end{split}\end{equation*}

\noindent After integrating by parts, and since $G(0)=0,$ we obtain 

\begin{equation}\label{eq:Po2}\begin{split}
& \frac{u'^2(R)G(R)}{2} +\int_0^R u'^2\left((n-1)G\coth(x) - \frac{G'}{2}\right)\, dx \\& = \lambda\int_0^R \frac{u^2}{2}G' \, dx + \frac{1}{p+1}\int_0^R {|u|}^{p+1}G'\, dx.\\
\end{split}
\end{equation}
Substituting equation (\ref{eq:Po1}) into equation (\ref{eq:Po2}), and since $1/2 - 1/(p+1) = 1/n,$ it follows that 
\begin{equation}\label{eq:Po}\begin{split}
&\int_0^R u'^2\left((n-1)G\coth(x) - \frac{G'}{2} - \frac{G'}{p+1}\right)\, dx +  \frac{u'^2(R)G(R)}{2} \\& = \frac{\lambda}{n}\int_0^R u^2 \sinh^{n-1}(x)\, dx.\\
\end{split}
\end{equation}
Notice that in equation (\ref{eq:cuociente}) we have written $\sinh^{n-1}(x)$ as $G'(x).$ Thus, since the boundary term is positive, and since $1/2 + 1/(p+1) = (n-1)/n,$ we have 
\begin{equation}\label{eq:cuociente}
\lambda \ge \frac{n(n-1)\displaystyle\int_0^R u'^2\left(G\coth(x) - \frac{G'}{n} \right)\, dx}{\displaystyle\int_0^R u^2 G'\, dx}.
\end{equation}
Now let $L(x) = G\coth(x) - \dfrac{G'}{n}.$ Then $L \ge 0.$ In fact, we can write $L(x) = m(x)/\sinh(x),$ where 
$$
m(x) = G\cosh(x) -\sinh^n(x)/n.
$$
Then, since $G(0),$ we have $m(0) = 0$.  Also, 
$$
m'(x)  = G\sinh (x) + G'\cosh(x) - \sinh^{n-1}(x)\cosh(x) = G\sinh(x).
$$
It follows that $m'\ge0,$ and therefore that $L\ge 0.$ 

\bigskip

We now use a Hardy type  inequality to write the denominator integral in terms of $u'^2$. For a review on Hardy's inequa;ities see, e.g., 
\cite{OpKu90,Da99}. Integrating by parts, we can write 
$$
\int_0^R u^2 G'\, dx = -2\int_0^R \left( u \sinh^{\frac{n-1}{2}}(x)\right) \left(\frac{Gu'}{\sinh^\frac{n-1}{2}}(x)\right)\, dx.
$$ 
Then, usiong Cauchy-Schwarz, it follows that 
$$ 
\left(\int_0^R u^2 G' \, dx \right)^2 \le 4\int_0^R u^2 G'\, dx \int_0^R \frac{G^2u'^2}{G'}\, dx.
$$ 
That is, 
\begin{equation}\label{eq:Hardy}
\int_0^R u^2 G' \, dx \le 4 \int_0^R \frac{u'^2G^2}{G'}\, dx.
\end{equation}
Using inequality (\ref{eq:Hardy}) in the quotient (\ref{eq:cuociente}),  we conclude that 
\begin{equation}\label{eq:cuoHarf}
\lambda \ge \frac{n(n-1)\displaystyle\int_0^R u'^2\left(G\coth(x) - \frac{G'}{n}\right)\, dx}{4 \displaystyle\int_0^R \frac{u'^2G^2}{G'}\, dx}.
\end{equation}
In the Lemma \ref{lem:desigualdad} below, we show that $L(x) \ge c \dfrac{G^2}{G'},$ where $c = \dfrac{n}{n+2}.$ With this, we conclude that 
\begin{equation}
\lambda \ge \frac{n^2(n-1)}{4(n+2)}.
\end{equation}
\end{proof}

\bigskip

\begin{lemma}\label{lem:desigualdad}
Let $x\ge 0$ and let $L(x) = G\coth(x) - \dfrac{G'}{n}.$ Then 
$$
L(x) \ge \frac{n}{(n+2)} \frac{G^2}{G'}.
$$ 
Here, as above, $G(x) = \displaystyle\int_0^x \sinh^{n-1}(s)\, ds.$ 
\end{lemma}

\begin{proof}
Let $f(x) = L(x)G'(x) - c\,G^2(x),$ where $c = \dfrac{n}{n+2}.$ It suffices to show that $f\ge 0.$ 

\bigskip

As before, we write $L(x) = \dfrac{m(x)}{\sinh(x)},$ where $m(x) = G\cosh(x) - \dfrac{\sinh^n(x)}{n}$ and $m'(x) = G\sinh(x).$ Then, 
$$
f(x) = \sinh^{n-2}(x)m(x) - c\,G^2(x).
$$
Notice that since $\sinh(0) = G(0) = 0$, one has that $f(0)= 0$, so it suffices to show that $f' \ge 0$. We have that 
$$ 
f'(x) = \sinh^{n-3}(x) \left((n-2) \cosh(x)m(x) + G\sinh^2(x)(1 - 2c)\right).
$$ 
Let $\sigma  = 2c-1 = 1/p,$ where $p = (n+2)/(n-2)$ is the critical Sobolev exponent; and let $g(x)= (n-2) \cosh(x)m(x) -\sigma G\sinh^2(x)$. 
It suffices to show that $g\ge 0.$ Since $m(0)=0,$ then $g(0) = 0.$ 
\bigskip 

\noindent
Also, 
\begin{equation*}
g'(x) = 2\sinh(x)\cosh(x) G(x) \left( n-2 - \sigma\right) - \sinh^{n+1}(x) \left(\frac{(n-2)}{n}+ \sigma \right)
\end{equation*}
and in particular $g'(0) = 0$. Since 
$$ 
n-2 - \sigma  = \frac{(n-2)(n+1)}{(n+2)}
$$
and 
$$ 
\frac{(n-2)}{n}+\sigma = \frac{2(n+1)(n-2)}{n(n+2)}
$$ 
we can write
$$
g'(x) = \frac{2(n+1)(n-2)}{n(n+2)}\sinh(x)\left[ nG\cosh(x) - \sinh^n(x)\right].
$$ 
Finally, let $h(x) = nG\cosh(x) - \sinh^n(x).$ If we show $h(x) \ge0$, then we will have $g' \ge 0$, which will imply $g \ge 0$, 
and thus, that $f\ge 0,$ as desired. Notice that $h(0)=0.$ Also, since $G'(x) = \sinh^{n-1}(x),$ we have
$$h'(x) = nG\sinh(x).$$ 
That is, $h'\ge 0,$ which concludes the proof of Lemma \ref{lem:desigualdad}.
\end{proof}

\begin{remark}
In the proof of Lemma \ref{lem:desigualdad}, the constant $\sigma = 1/p$ plays a crucial role, where $p$ is the critical Sobolev exponent. It is worth noting that for small $x$ and $g$ as in the proof above, 
$$ 
g(x)   = x^{n+2} \left(\frac{1}{np} - \frac{\sigma}{n} \right) + \mathcal{O}(x^{n+4}).
$$ 
It follows that if $\sigma \le 1/p, $ then $g $ is positive in a neighborhood of the origin. 
It was this observation that led us to realize that $\sigma = 1/p$ would yield the optimal estimate.  
\end{remark}

\section{References}

\bigskip

\renewcommand{\refname}{}    
\vspace*{-36pt}              


\frenchspacing
\begin{thebibliography}{7}


\bibitem{Ba75}
E.~Balslev, Absence of positive eigenvalues of Schr\"odinger Operators, 
{\it Archive~Rational~Mechanics and Applications}, {\textbf 59} (1975), 343--357. 


\bibitem{BaKa12}
C. Bandle and Y. Kabeya, On the positive, radial solutions of a semilinear elliptic equation in $\mathbb{H}^N.$ {\it Adv. Nonlinear Anal.,} {\textbf 1} (2012), 1--25.

\bibitem{Be16}
S. Benguria, The solution gap of the Brezis-Nirenberg problem on the hyperbolic space. {\it Monatsh. Math.}, DOI 10.1007/s00605-015-0861-1 (2016).



\bibitem{BrNi83}
H. Br\'ezis and L. Nirenberg, Positive solutions of nonlinear elliptic equations involving critical Sobolev exponents. {\it Comm. Pure Appl. Math.}, {\textbf 36} (1983), 437--477.

\bibitem{Da99}
E.~B.~Davies, A review on Hardy inequalities.
{\it Operator Theory Advances and Applications}, {\textbf 110}, pp.~55--67, Birkh\"auser Verlag, Basel, 1999.

\bibitem{GaSa14}
D. Ganguly and K. Sandeep, Sign changing solutions of the Brezis-Nirenberg problem in
the hyperbolic space. {\it Calc. Var. Partial Differential Equations}, {\textbf 50} (1-2) (2014), 69--91.

\bibitem{OpKu90}
B.~Opic and A.~Kufner, Hardy--type inequalities. 
Pitman Research Notes in Math.,
{\textbf 219}, Longman, 1990.


\bibitem{Po65}
S. I. Pohozaev, On the eigenfunctions of the equation $\Delta u+\lambda f(u)=0. $ {\it Dokl. Akad. Nauk.}, {\textbf  165} (1965), 36--39. 

\bibitem{Re40}
F.~Rellich, Darstellung der Eigenwerte von $\Delta u + \lambda  u = 0$ durch ein Randintegral, {\it Math. Z.}, {\textbf 46} (1940), 635--636.

\bibitem{Si74}
B.~Simon, Absence of positive eigenvalues in a class of multiparticle quantum systems, 
{\it Math.~Ann.}, {\textbf 207} (1974), 133--138.


\bibitem{St02}
S. Stapelkamp, The Br\'ezis-Nirenberg problem on $\mathbb{H}^n.$ Existence and uniqueness of solutions.  {\it Elliptic
and parabolic problems} (Rolduc/Gaeta, 2001), World Sci. Publ., River Edge, NJ 2002, 283--290.

\bibitem{We67}
J.~Weidmann, The virial theorem and its application to the spectral theory of Schr\"odinger operators, 
{\it Bull.~Amer.~Math.~Soc.} {\textbf 77} (1967), 452--456.


\end{thebibliography}
\end{document}